\documentclass[14pt]{amsart}
\usepackage{cases}
\usepackage{amsmath}
\usepackage{amsfonts}
\usepackage{bm}
\usepackage{amsfonts,amsmath,amssymb,amscd,bbm,amsthm,mathrsfs,dsfont}
\usepackage{mathrsfs}
\usepackage{pb-diagram}
\usepackage{color}
\usepackage{amssymb}
\usepackage{xypic}
\usepackage[all]{xy}
\usepackage{mathrsfs}
\usepackage{amsthm}
\usepackage[all]{xy}
\usepackage{epsfig}
\usepackage{graphics}
\usepackage{array}
\usepackage{graphicx} 
\usepackage{epstopdf}
\usepackage{float}
\usepackage{tikz}
\usepackage{tikz-cd}
\usetikzlibrary{arrows}
\usetikzlibrary{graphs}
\usepackage{tikz}
\usepackage{hyperref}
\usepackage{CJK}

\newtheorem{Theorem}{Theorem}[section]
\newtheorem{Lemma}[Theorem]{Lemma}
\newtheorem{Definition}[Theorem]{Definition}

\newtheorem{Conjecture}[Theorem]{Conjecture}

\title [A conjecture on cluster automorphisms of cluster algebras]{A conjecture on cluster automorphisms of \\cluster algebras}

\author{Peigen Cao}
\author{Fang Li}
\author{Siyang Liu}
\author{Jie Pan}
\address{Peigen Cao
\newline Department of Mathematics, Zhejiang University (Yuquan Campus), Hangzhou, Zhejiang
310027,  P.R.China}
\email{peigencao@126.com}

\address{Fang Li
\newline Department of Mathematics, Zhejiang University (Yuquan Campus), Hangzhou, Zhejiang
310027,  P.R.China}
\email{fangli@zju.edu.cn}

\address{Siyang Liu
\newline Department
of Mathematics, Zhejiang University (Yuquan Campus), Hangzhou, Zhejiang
310027, P.R.China}
\email{siyangliu@zju.edu.cn}

\address{Jie Pan
\newline Department
of Mathematics, Zhejiang University (Yuquan Campus), Hangzhou, Zhejiang
310027, P.R.China}
\email{panjiezhejiang@qq.com}

\date{version of \today}

\newcommand{\lra}{\longrightarrow}

\newcommand{\ra}{\rightarrow}
\newcommand{\sdp}{\times\kern-.2em\vrule height1.1ex depth-.05ex}
\newcommand{\epi}{\lra \kern-.8em\ra}

\makeatletter
\@addtoreset{equation}{section}
\makeatother

 \normalbaselines
\begin{document}

\begin{CJK*}{GBK}{song}
\renewcommand{\thefootnote}{\alph{footnote}}
\setcounter{footnote}{-1} 
\renewcommand{\thefootnote}{\alph{footnote}}
\setcounter{footnote}{-1} 

\maketitle
\bigskip
\begin{abstract}
A cluster automorphism is a $\mathbb{Z}$-algebra automorphism of a cluster algebra $\mathcal A$ satisfying that  it sends a cluster to another and  commutes  with mutations. Chang and Schiffler conjectured that a cluster automorphism of $\mathcal A$  is just a $\mathbb{Z}$-algebra homomorphism of a cluster algebra sending a cluster to another.  The aim of this article is to prove this conjecture.
\end{abstract}

\section{Introduction}

Cluster algebras were invented by Fomin and Zelevinsky in a series of papers \cite{FZ1, BFZ, FZ2, FZ4}. A cluster algebra  is a $\mathbb Z$-subalgebra of an ambient field $\mathcal F=\mathbb Q(u_1,\cdots,u_n)$ generated by certain combinatorially defined generators (i.e.,  {\em cluster variables}), which are grouped into overlapping {\em clusters}. Many relations between cluster
algebras and other branches of mathematics have been discovered, for example, Poisson geometry, discrete dynamical systems, higher Teichm\"uller spaces, representation theory of quivers and finite-dimensional algebras.

We first recall the definition of cluster automorphisms, which were introduced by Assem,
Schiffler and Shamchenko in \cite{ASS}.
\begin{Definition}[\cite{ASS}]\label{cluauto} Let $\mathcal{A}=\mathcal{A}(\mathbf{x},B)$ be a cluster algebra, and $f:\mathcal{A}\rightarrow \mathcal{A}$
be an automorphism of $\mathbb{Z}$-algebras. $f$ is called a {\bf cluster automorphism} of $\mathcal A$ if there exists another seed $(\mathbf{z},B')$ of $\mathcal{A}$ such that

\begin{enumerate}
\item[(1)]$f({\bf x})={\bf z}$;
\item[(2)] $f(\mu_x({\bf x}))=\mu_{f(x)}({\bf z})$ for any $x\in{\bf x}$.
\end{enumerate}

\end{Definition}

Cluster automorphisms and their related groups were studied by many authors, one can refer to \cite{CZ0,CZ1,CZ2,LL,HLY,CL3,CS} for details.

The following very insightful conjecture on cluster automorphisms is by Chang and Schiffler, which suggests that we can  weaken the conditions in Definition \ref{cluauto}. In particular, it suggests that the second condition in Definition \ref{cluauto} can be obtained from the first one and the assumption that $f$ is a $\mathbb{Z}$-algebra homomorphism.

\begin{Conjecture}\cite[Conjecture 1]{CS} \label{conj} Let $\mathcal A$ be a cluster algebra, and $f:\mathcal A\rightarrow \mathcal A$ be a $\mathbb{Z}$-algebra homomorphism. Then $f$ is
cluster automorphism if and only if there exist two clusters $\mathbf{x}$ and $\mathbf{z}$ such that $f({\bf x})={\bf z}$.
\end{Conjecture}

The following is our main result, which affirms the Conjecture \ref{conj}.

{\bf Theorem \ref{main2}}\;
{\em    Let $\mathcal A$ be a cluster algebra, and $f:\mathcal A\rightarrow \mathcal A$ be a $\mathbb{Z}$-algebra homomorphism. Then $f$ is a
cluster automorphism if and only if there exist two clusters $\mathbf{x}$ and $\mathbf{z}$ such that $f({\bf x})={\bf z}$.
  }

\section{Preliminaries}

In this section, we recall basic concepts and important properties of cluster algebras. In this paper, we focus
on cluster algebras without coefficients  (that is, with trivial coefficients).
For a positive integer $n$, we will always denote by $[1, n]$ the set $\{1,2,\dots, n\}$.

Recall that $B$ is said to be {\bf skew-symmetrizable} if there exists an positive  diagonal integer
matrix $D$ such that $BD$ is skew-symmetric.

Fix an ambient field $\mathcal{F} = \mathbb{Q}(u_1, u_2, \dots, u_n)$. A {\bf labeled seed} is a pair $(\mathbf{x}, B)$, where $\mathbf{x}$ is an $n$-tuple of free generators of $\mathcal{F}$, and $B$ is an $n\times n$ skew-symmetrizable integer matrix.
For $k\in [1,n]$, we can define another pair
$(\mathbf{x'},B') = \mu_k(\mathbf{x},B)$, where
\begin{enumerate}
\item[(1)] $\mathbf{x'} = (x'_1, \dots, x'_n)$ is given by
 \[x'_k = \frac{\prod\limits_{i=1}^n x_i^{[b_{ik}]_+} + \prod\limits_{i=1}^n x_i^{[-b_{ik}]_+}}{x_k}\]  and $x'_i = x_i$ for $i \neq k$;
\item[(2)] $B'=\mu_k(B)=(b'_{ij})_{n\times n}$ is given by \[b'_{ij}=\begin{cases} -b_{ij},&\text{if $i=k$ or $j=k$;}\\b_{ij} + \mathrm{sgn}(b_{ik})[b_{ik}b_{kj}]_+,&\text{otherwise.}\end{cases}\]
\end{enumerate}
where $[x]_+=\mathrm{max}\{x,0\}$. The new pair $(\mathbf{x'},B') = \mu_k(\mathbf{x},B)$ is called  the {\bf mutation} of $(\mathbf{x},B)$
 at $k$. We also denote $B^\prime=\mu_k(B)$.

It can be seen that $(\mathbf{x'},B')$ is also a labeled
seed and $\mu_k$ is an involution.

Let $(\mathbf{x},B)$ be a labeled seed. $\mathbf{x}$
is called a {\bf labeled cluster}, elements in $\mathbf{x}$ are called  {\bf cluster variables}, and $B$ is
called an  {\bf exchange matrix}. The {\bf unlabeled seeds} are obtained by identifying labeled seeds that differ from each other by simultaneous permutations
of the components in $\mathbf{x}$, and of the rows and columns of $B$. We will refer to unlabeled seeds and unlabeled clusters simply as {\bf seeds} and {\bf clusters} respectively, when there is no risk of confusion.

\begin{Lemma}[\cite{BFZ}]\label{le}
Let $B$ be an $n\times n$ skew-symmetrizable matrix. Then $\mu_k(B) = (J_k+E_k)B(J_k+F_k)$, where
\begin{enumerate}
\item $J_k$ denotes the diagonal $n\times n$ matrix whose
diagonal entries are all $1'$s, except for $-1$ in the $k$-th position;
\item $E_k$ is the $n\times n$ matrix whose only nonzero entries are $e_{ik} = [-b_{ik}]_+$;
\item $F_k$ is the $n\times n$ matrix whose only nonzero entries are $f_{kj} = [b_{kj}]_+$.
\end{enumerate}
\end{Lemma}

\begin{Definition}[\cite{FZ1,FZ4}]
\begin{enumerate}
\item[(1)] Two labeled seeds  $(\mathbf{x},B)$ and $(\mathbf{x}',B')$ are said to be  {\bf mutation equivalent} if $(\mathbf{x}',B')$ can be obtained from $(\mathbf{x},B)$ by a sequence of mutations;

\item[(2)] Let $\mathrm{T}_n$ be an $n$-regular tree and valencies emitting from each vertex are labelled by $1,2,\dots,n$.
A {\bf cluster pattern} is an $n$-regular tree $\mathrm{T}_n$ such that for each vertex $t\in \mathrm{T}_n$, there is a
labeled seed $\Sigma_t=(\mathbf{x}_t,B_t)$ and for each edge labelled by $k$, two labeled seeds in the endpoints
are obtained from each other by seed mutation at $k$. We always write
\[\mathbf{x}_t=(x_{1,t},x_{2,t},\dots,x_{n,t}),\,\,\,\,\,\,B_t=(b_{ij}^t).\]
\end{enumerate}
\end{Definition}

The cluster algebra $\mathcal{A}=\mathcal{A}(\mathbf{x}_{t_0},B_{t_0})$
associated with the initial seed $(\mathbf{x}_{t_0},B_{t_0})$ is a $\mathbb{Z}$-subalgebra of $\mathcal{F}$
generated by cluster variables appeared in $T_n(\mathbf{x}_{t_0},B_{t_0})$, where $T_n(\mathbf{x}_{t_0},B_{t_0})$ is the cluster pattern with $(\mathbf{x}_{t_0},B_{t_0})$ lying in
the vertex $t_0 \in T_n$.

\begin{Theorem}[Laurent phenomenon and positivity \cite{FZ4,LS,GHKK}]
Let $\mathcal{A}=\mathcal{A}(\mathbf{x}_{t_0}, B_{t_0})$ be the cluster algebra. Then  each cluster variable $x_{i,t}$ is contained in $\mathbb{Z}_{\geqslant 0}[x_{1;t_0}^{\pm 1},x_{2;t_0}^{\pm1},\dots,x_{n;t_0}^{\pm1}]$.
\end{Theorem}

\section{The proof of main result}

In this section, we will give our  main result, which affirms the Conjecture \ref{conj}.

\begin{Lemma}\label{le1}
Let $0\neq B$ be a skew-symmetrizable integer matrix. If $B^\prime$ is obtained from $B$ by a sequence of mutations and $B=aB^\prime$ for some $a\in\mathbb Z$, then $a=\pm 1$ and $B=\pm B^\prime$.
\end{Lemma}
\begin{proof}

Since $B^\prime$ is obtained from $B$ by a sequence of mutations, there exist integer matrices $E$ and $F$ such that $B^\prime=EBF$, by Lemma \ref{le}.
If $B=aB^\prime$, then we get $B^\prime=aE(B^\prime)F$. Also, we can have

$$B^\prime=aE(B^\prime)F=a^2E^2(B^\prime)F^2=\cdots=a^sE^s(B^\prime)F^s,$$
where $s\geq 0$. Thus  $\frac{1}{a^s}B^\prime=E^s(B^\prime)F^s$ holds for any $s\geq0$.

Assume by contradiction that $a\neq \pm1$, then when $s$ is large enough, $\frac{1}{a^s}B^\prime$ will not be an integer matrix. But $E^s(B^\prime)F^s$ is always an integer matrix. This is a contradiction. So we must have $a=\pm 1$ and thus  $B=\pm B^\prime$.
\end{proof}

A square matrix $A$ is {\bf decomposable} if there exists a permutation matrix $P$ such
that $PAP^{\rm T}$ is a block-diagonal matrix, and {\bf indecomposable} otherwise.

\begin{Lemma}\label{lemma1}
Let $0\neq B$ be an indecomposable skew-symmetrizable matrix. If $B^\prime$ is obtained from $B$ by a sequence of mutations and  $B=B^\prime A$ for some integer diagonal matrix $A=diag(a_1,\cdots,a_n)$, then $A=\pm I_n$ and $B=\pm B^\prime$.
\end{Lemma}
\begin{proof}
If there exists $i_0$ such that $a_{i_0}=0$, then the $i_0$-th column vector of $B$ is zero, by $B=B^\prime A$. This contradicts that $B$ is indecomposable and $B\neq 0$. So each $a_{i_0}$ is nonzero for $i_0=1,\cdots,n$.

Let $D=diag(d_1,\cdots,d_n)$ be a skew-symmetrizer of $B$. By $B=B^\prime A$ and $AD=DA$, we know that $$BD=B^\prime AD=(B^\prime D)A.$$

By the definition of mutation, we know that $D$ is also a skew-symmetrizer of $\mu_k(B)$, $k=1,\cdots,n$. Since $B^\prime$ is obtained from $B$ by a sequence of mutations, we get that $D$ is a skew-symmetrizer of $B^\prime$. Namely, we have that both  $B^\prime D$ and $BD=(B^\prime D)A$ are skew-symmetric.
Since $0\neq B$ is indecomposable, we must have $a_1=\cdots=a_n$. So $A=aI_n$ for some $a\in\mathbb Z$, and $B=aB^\prime$. Then by Lemma \ref{le1}, we can get $A=\pm I_n$ and $B=\pm B^\prime$.
\end{proof}

\begin{Lemma}\label{lemma2}
Let $B=diag(B_1,\cdots,B_s)$, where each $B_i$ is a nonzero indecomposable skew-symmetrizable matrix of size $n_i\times n_i$. If $B^\prime$ is obtained from $B$ by a sequence of mutations and  $B=B^\prime A$ for some integer diagonal  matrix $A=diag(a_{1},\cdots,a_{n})$, then $a_j=\pm 1$ for $j=1,\cdots,n$.
\end{Lemma}
\begin{proof}
By the definition of mutation, we know that $B^\prime$ has the form of $B^\prime=diag(B_1^\prime,\cdots,B_s^\prime)$, where each $B_i^\prime$ is obtained from $B_i$ by a sequence of mutations. We can write $A$ as a block-diagonal matrix $A=diag(A_1,\cdots,A_s)$, where $A_i$ is a $n_i\times n_i$ integer diagonal matrix. By $B=B^\prime A$, we know that $B_i=B_i^\prime A_i$. Then by Lemma \ref{lemma1}, we have $A_i=\pm I_{n_i}$ and $B_i=\pm B_i^\prime$ for $i=1,\cdots,s$.
In particular, we get $a_j=\pm 1$ for $j=1,\cdots,n$.
\end{proof}

\begin{Lemma}\label{le2}
Let $\mathcal A=\mathcal A(\mathbf{x}, B)$ be a cluster algebra, and $f : \mathcal{A} \rightarrow  \mathcal{A} $ be a $\mathbb Z$-homomorphism of $\mathcal{A}$. If there exists another seed $(\mathbf{z},B')$ of $\mathcal{A}$ such that such that $f({\bf x})={\bf z}$, then  $f(\mu_x({\bf x}))=\mu_{f(x)}({\bf z})$ for any $x\in{\bf x}$.
\end{Lemma}
\begin{proof}
After permutating the rows and columns of $B$, it can be written as a block-diagonal matrix as follows.
 $$B=diag(B_1,B_2,\cdots,B_s),$$
  where $B_1$ is an $n_1\times n_1$ zero matrix and $B_j$ is nonzero indecomposable skew-symmetrizable matrix of size $n_j\times n_j$ for $j=2,\cdots,s$.

Without loss of generality, we  assume that
$f(x_i) = z_i $ for $1\leq i \leq n$.

Let $x_k^\prime$ and $z_k^\prime$ be the new obtained variables in $\mu_k(\mathbf{x}, B)$ and $\mu_k(\mathbf{z}, B')$. So we have

$$x_kx_k^\prime=\prod\limits_{i=1}^nx_i^{[b_{ik}]_+}+\prod\limits_{i=1}^nx_i^{[-b_{ik}]_+},\;\;\;\text{and}\;\;\;
z_kz_k^\prime=\prod\limits_{i=1}^nz_i^{[b_{ik}^\prime]_+}+\prod\limits_{i=1}^nz_i^{[-b_{ik}^\prime]_+}.$$
Thus
\begin{eqnarray}
f(x_k^\prime) &=& f\left(\frac{\prod\limits_{i=1}^nx_i^{[b_{ik}]_+}+\prod\limits_{i=1}^nx_i^{[-b_{ik}]_+}}{x_k}\right)
\nonumber\\
&=& \frac{\prod\limits_{i=1}^nz_i^{[b_{ik}]_+}+\prod\limits_{i=1}^nz_i^{[-b_{ik}]_+}}{z_k}\nonumber \\
      &=& \frac{\prod\limits_{i=1}^nz_i^{[b_{ik}]_+}+\prod\limits_{i=1}^nz_i^{[-b_{ik}]_+}}
      {\prod\limits_{i=1}^nz_i^{[b_{ik}^\prime]_+}+\prod\limits_{i=1}^nz_i^{[-b_{ik}^\prime]_+}} z'_k.\nonumber
      \end{eqnarray}
      Note that the above expression is the expansion of $f(x_k^\prime)$ with respect to the cluster $\mu_k({\bf z})$.
By $$f(x_k^\prime) \in f(\mathcal A)=\mathcal{A} \subset \mathbb Z[z_1^{\pm 1}, \dots , (z'_k)^{\pm 1}, \dots , z_n^{\pm 1}],$$ we can get
 $$\frac{\prod\limits_{i=1}^nz_i^{[b_{ik}]_+}+\prod\limits_{i=1}^nz_i^{[-b_{ik}]_+}}
      {\prod\limits_{i=1}^nz_i^{[b_{ik}^\prime]_+}+\prod\limits_{i=1}^nz_i^{[-b_{ik}^\prime]_+}}\in\mathbb Z[z_1^{\pm1},\cdots,z_{k-1}^{\pm1},z_{k+1}^{\pm1},\cdots,z_n^{\pm1}].$$
     Since both $\prod\limits_{i=1}^nz_i^{[b_{ik}]_+}+\prod\limits_{i=1}^nz_i^{[-b_{ik}]_+}$ and $\prod\limits_{i=1}^nz_i^{[b_{ik}^\prime]_+}+\prod\limits_{i=1}^nz_i^{[-b_{ik}^\prime]_+}$ is not divided by any $z_i$, we actually have
 $$\frac{\prod\limits_{i=1}^nz_i^{[b_{ik}]_+}+\prod\limits_{i=1}^nz_i^{[-b_{ik}]_+}}
      {\prod\limits_{i=1}^nz_i^{[b_{ik}^\prime]_+}+\prod\limits_{i=1}^nz_i^{[-b_{ik}^\prime]_+}}\in\mathbb Z[z_1,\cdots,z_{k-1},z_{k+1},\cdots,z_n].$$
   So for each $k$, there exists an integer $a_k\in \mathbb{Z}$ such that $(b_{1k}, b_{2k}, \dots , b_{nk})^{\rm T} = a_k(b'_{1k}, b'_{2k}, \dots , b'_{nk})^{\rm T}$.
 Namely, we have $B=B^\prime A$, where $A=diag(a_1,\cdots,a_n)$.
  Note that $B$ has the form of $$B=diag(B_1,B_2,\cdots,B_s),$$ where $B_1$ is a $n_1\times n_1$ zero matrix and $B_j$ is nonzero indecomposable skew-symmetrizable matrix of size $n_j\times n_j$ for $j=2,\cdots,s$.
 Applying Lemma \ref{lemma2} to the skew-symmetrizable matrix $diag(B_2,\cdots,B_s)$, we can get $a_{j}=\pm 1$ for $n_1+1,\cdots,n$.
 Since the first $n_1$ column vectors of both $B$ and $B^\prime$ are zero vectors, we can just take $a_1=\cdots=a_{n_1}=1$. So for each $k$, we have $a_k=\pm1$ and
 $$(b_{1k}, b_{2k}, \dots , b_{nk})^{\rm T} = a_k(b'_{1k}, b'_{2k}, \dots , b'_{nk})^{\rm T}=\pm (b'_{1k}, b'_{2k}, \dots , b'_{nk})^{\rm T}.$$

Hence, $$\frac{\prod\limits_{i=1}^nz_i^{[b_{ik}]_+}+\prod\limits_{i=1}^nz_i^{[-b_{ik}]_+}}
      {\prod\limits_{i=1}^nz_i^{[b_{ik}^\prime]_+}+\prod\limits_{i=1}^nz_i^{[-b_{ik}^\prime]_+}}=1.$$
      Thus we get
\begin{eqnarray}
f(x_k^\prime)=\frac{\prod\limits_{i=1}^nz_i^{[b_{ik}]_+}+\prod\limits_{i=1}^nz_i^{[-b_{ik}]_+}}
      {\prod\limits_{i=1}^nz_i^{[b_{ik}^\prime]_+}+\prod\limits_{i=1}^nz_i^{[-b_{ik}^\prime]_+}} z_k^\prime=z_k^\prime.\nonumber
      \end{eqnarray}
      So  $f(\mu_x({\bf x}))=\mu_{f(x)}({\bf z})$ for any $x\in{\bf x}$.
\end{proof}

\begin{Lemma}\label{le3}
Let $\mathcal A=\mathcal{A}({\bf x}, B)\subseteq\mathcal F$ be a cluster algebra, and $f$ be an automorphism of the ambient field $\mathcal F$. If there exists another seed $(\mathbf{z},B')$ of $\mathcal{A}$ such that $f({\bf x})={\bf z}$ and $f(\mu_x({\bf x}))=\mu_{f(x)}({\bf z})$ for any $x\in{\bf x}$. Then

(i) $f$ is an automorphism of $\mathcal{A}$;

(ii) $f$ is a  cluster automorphism of $\mathcal{A}$.
\end{Lemma}
\begin{proof}
(i) Since $f$ is an  automorphism of the ambient field $\mathcal F$, we know that $f$ is injective.

Since $f$ commutes with mutations, we know that  $f$ restricts to
a surjection on $\mathcal{X}$, where $\mathcal{X}$ is the set of cluster variables of $\mathcal A$. Because $\mathcal{A}$ is generated by  $\mathcal{X}$, we get that $f$ restricts to a epimorphism of $\mathcal{A}$.

Hence, $f$ is an automorphism of $\mathcal A$.

(ii) follows from (i) and the definition of cluster automorphisms.
\end{proof}

\begin{Theorem}\label{main2}
 Let $\mathcal A$ be a cluster algebra, and $f:\mathcal A\rightarrow \mathcal A$ be a $\mathbb{Z}$-algebra homomorphism. Then $f$ is a
cluster automorphism if and only if there exist two clusters $\mathbf{x}$ and $\mathbf{z}$ such that $f({\bf x})={\bf z}$.
\end{Theorem}
\begin{proof}
``{\em Only if part}": It follows from the definition of cluster automorphism.

 ``{\em If part}": It follows from Lemma \ref{le2} and Lemma \ref{le3}.
\end{proof}

\vspace{5mm}
{\bf Acknowledgements:}\; This project is supported by the National Natural Science Foundation of China (No.11671350 and No.11571173) and the Zhejiang Provincial Natural Science Foundation of China (No.LY19A010023).

\end{CJK*}

\vspace{10mm}
\begin{thebibliography}{99}


\bibitem{ASS}
I. Assem, R. Schiffler, V. Shamchenko, Cluster automorphisms. Proc. Lond. Math. Soc.,
(3)104 (2012),1271-1302

\bibitem{BFZ}
A. Berenstein, S. Fomin, A. Zelevinsky, Cluster algebras III: upper bounds and double Bruhat cells. Duke Math. J. (1)126 (2005), 1-52


\bibitem{CL1}
P. Cao, F. Li, The enough $g$-pairs property and denominator vectors of cluster algebras.
arXiv:1803.05281v2

\bibitem{CL3}
P. Cao, F. Li, Unistructurality of cluster algebras, preprint, arXiv:1809.05116.

\bibitem{CS}
W. Chang, R. Schiffler, A note on cluster automorphism groups, axXiv: 1812.05034.

\bibitem{CZ0}
W. Chang, B. Zhu, Cluster automorphism groups of cluster algebras with coefficients, Sci. China Math. (10)59(2016), 1919-1936

\bibitem{CZ1}
W. Chang, B. Zhu, Cluster automorphism groups and automorphism groups of exchange graphs. arxiv: 1506.02029

\bibitem{CZ2}
W. Chang, B. Zhu, Cluster automorphism groups of cluster algebras of finite type. J. Algebra.
447(2016):490-515

\bibitem{FZ1}
S. Fomin, A. Zelevinsky, Cluster algebras I: Foundations. J. Amer. Math. Soc. (2)15(2002), 497-529

\bibitem{FZ2}
S. Fomin, A. Zelevinsky, Cluster algebras II: Finite type classification. Invent. Math. (1)154(2003), 63-121



\bibitem{FZ4}
S. Fomin, A. Zelevinsky, Cluster algebras IV: Coefficients. Compos. Math. (1)143(2007), 112-164


\bibitem{GHKK}
M. Gross, P. Hacking, S. Keel, M. Kontsevich, Canonical bases for cluster algebras,
2014,arxiv:1411.1394

\bibitem{HLY}
M. Huang, F. Li, and Y. Yang, On structure of sign-skew-symmetric cluster algebras of geometric type,
I: In view of sub-seeds and seed homomorphisms, Sci. China Math. (5)61(2018), 831-854

\bibitem{LL}
F. Li, S. Liu. Periodicities in cluster algebras and cluster automorphism groups.
arXiv:1903.00893

\bibitem{LS}
K. Lee, R. Schiffler. Positivity for cluster algebras.
Ann. of Math. (2) 182(2015), no. 1, 73-125

\bibitem{S}
I. Saleh. Exchange maps of cluster algebras. Int. Electron. J. Algebra 16 (2014), 1-15
\end{thebibliography}
\end{document}